\documentclass[11pt]{35}
\usepackage{amsthm, amsmath, amssymb, amsbsy, amsfonts, latexsym, euscript,amscd}
\usepackage{calrsfs}
\usepackage{graphicx}

\usepackage{algorithm}%
\usepackage{algorithmic}%

\def\le{\leqslant}

\def\ge{\geqslant}
\def\geq{\geqslant}

\def\phi{\varphi}

\def\bar{\overline}

\def\kappa{\varkappa}

\newtheorem{theorem}{\bf \indent Theorem}[section]

\newtheorem{lemma}{\bf \indent Lemma}[section]

\theoremstyle{remark}
\newtheorem{remark}{\bf \indent Remark}[section]

\numberwithin{equation}{section}

\begin{document}

\vskip5mm
\noindent
{\Large \bf
STOCHASTIC ORIGIN FRANK-WOLFE FOR 
\\[3pt]
TRAFFIC ASSIGNMENT
}

\vskip7mm

{\bf Igor Ignashin}$^*$
\vskip-2pt
{\small Moscow Institute of Physics and Technology}
\vskip-4pt
{\small 9, Institutskiy per., Dolgoprudny 141700, Russia}
\vskip-2pt
{\small
ignashin.in@phystech.edu}

\vskip7pt

{\bf Demyan Yarmoshik}
\vskip-2pt
{\small Moscow Institute of Physics and Technology}
\vskip-4pt
{\small 9, Institutskiy per., Dolgoprudny 141700, Russia}
\vskip-2pt
{\small
yarmoshik.dv@phystech.edu}

\vskip7pt

{\bf Andrei Raigorodskii}
\vskip-2pt
{\small Moscow Institute of Physics and Technology}
\vskip-4pt
{\small 9, Institutskiy per., Dolgoprudny 141700, Russia}
\vskip-2pt
{\small
mraigor@yandex.ru}

\vskip10mm

\parbox{146mm}{\noindent\it

In this paper, we present the Stochastic Origin Frank-Wolfe (SOFW) method, 
which is a special case of the block-coordinate Frank-Wolfe algorithm, applied to the problem 
of finding equilibrium flow distributions. By significantly reducing the computational complexity 
of the minimization oracle, the method improves overall efficiency at the cost of increased 
memory consumption. Its key advantage lies in minimizing the number of shortest path computations. 

We refer to existing theoretical convergence guarantees for generalized coordinate Frank-Wolfe methods 
and, in addition, extend the analysis by providing a convergence proof for a batched version of 
the Block-Coordinate Frank-Wolfe algorithm, which was not covered in the original work. 
We also demonstrate the practical effectiveness of our approach through experimental results. 
In particular, our findings show that the proposed method significantly outperforms the classical 
Frank-Wolfe algorithm and its variants on large-scale datasets. On smaller datasets, SOFW also 
remains effective, though the performance gap relative to classical methods becomes less pronounced. 
In such cases, there is a trade-off between solution quality, iteration time complexity, and memory usage.

}

{\centering
\section{Introduction}
}

Predicting flows in urban transportation networks is an essential problem for urban planing and traffic management.
Short-term (seconds to hours) flow prediction is usually approached with machine-learning techniques~\cite{jiang2022graph}. 
For making decisions on mid and long-term timescale (days to decays), engineers apply mathematical models which can roughly be classified into static models, dynamic models and multi-agent simulation models.
Static equilibrium transportation models play a special role due to their simplicity, computational efficiency and the abundant research and civil practice related to them~\cite{blubook-vol1-v091}.
They are still the default choice for multistage models of large networks, where traffic demand is also estimated  within the model, resulting in multiple evaluations of flow distribution in single model run~\cite{kubentayeva2024primal}.

Static equilibrium models are mainly represented by Beckmann's model~\cite{beckmann1956studies} and a less popular model of Nesterov and De-Palma~\cite{nesterov2003stationary}.
The first model is equivalent to a nonlinear convex min-cost multicommodity flow problem (MCF), while the latter is a linear capacitated min-cost MCF.

For both models a lot of attention was attracted to the design of efficient algorithms for searching the equilibria.
One of the effective methods for finding equilibrium in Beckmann's model is the link-based Frank-Wolfe or Frank-Wolfe method \cite{frank1956algorithm,levitin1966constrained}. 
Existing modifications, such as Conjugate Frank-Wolfe~\cite{mitradjieva2013stiff} or N-Conjugate Frank-Wolfe~\cite{ignashin2024modifications}, provide a significant advantage over the basic variant.
More advanced path-based and bush-based algorithms were designed to achieve solutions with higher accuracy, but are associated with such drawbacks as increased memory usage or complicated implementation~\cite{perederieieva2015framework,babazadeh2020reduced}.
Distinct from that are approaches based on solving the dual problem for both Beckmann's and Nesterov--De-Palma's models \cite{gasnikov2016numericalbeckman,kubentayeva2021finding} and the approach based on saddle-point formulation~\cite{yarmoshik2024application,zhang2025solving}.

All algorithmic approaches except the saddle-point approach require finding the shortest path between each origin-destination pair at each iteration. 
This is typically done by calling the Dijkstra's algorithm from each origin node.
Thus the complexity of each iteration grows with the size of network as $O(sn \log n)$, where $n$ is the number of nodes of the road graph and $s$ is the number of origin vertices, what makes it prohibitive to consider highly-detalized models where the number of origins approaches the number of all nodes in the graph.

Note that a similar difficulty in machine learning was bypassed by using stochastic gradient descent method (SGD) which generates randomized step directions using small portions of (large) dataset instead of computing exact gradient over all training samples at each iteration.
The idea of using stochastic dual gradient methods for Beckmann's model was described in~\cite{gasnikov2016numericalbeckman}, but was never validated experimentally. E.g., derivative works \cite{kubentayeva2021finding,kubentayeva2024primal} with numerical results only used deterministic algorithms. 

This paper studies stochastic variants of Frank-Wolfe and universal primal-dual algorithm for finding equilibria in Beckmann's model. 
We propose Stochastic Origin Frank-Wolfe methods (SOFW). This method is a special case of Block-Coordinate Frank-Wolfe algorithm proposed in \cite{jaggi2012blockcoordfw}. The idea behind this method is to find a Frank-Wolfe step on a subset of correspondences. 
The idea of the second method is to find a new direction conjugate to N previous directions.

The key findings of this paper are as follows:
\begin{enumerate}
    \item We propose the SOFW method, a block-coordinate variant of Frank-Wolfe with reduced oracle complexity via sparse flow decompositions.
    \item A weighted variant (SOFW-w) improves initial convergence by prioritizing important updates.
    \item The method introduces a memory–computation trade-off: more memory allows fewer shortest path calls.
    \item Experiments on real-world datasets show that SOFW significantly outperforms Frank-Wolfe on large graphs.
    \item On small graphs, SOFW remains competitive, though the performance gap is smaller.
    \item SOFW-w converges faster initially, but aligns with SOFW over time.
\end{enumerate}


{\centering
\section{Problem statement}
}

In this paper we consider the problem of finding an equilibrium distribution of flows in Beckmann's model, which is equivalent to the following convex optimization problem:
\begin{gather*}
    \Psi(f) = \sum_{e \in \mathcal{E}} \underbrace{\int_{0}^{f_e} \tau_e (z) d z}_{\sigma_e(f_e)} 
    \longrightarrow \min_{f \in \mathcal F} .
\end{gather*}

Let us introduce the notation.
The network is represented by a directed graph $\mathcal G = (\mathcal E, \mathcal V)$, where edges corresponds to roads.

Beckmann's model assume that travel time on a road is a monotonically increasing function of flow on the road.
Usually, Bureau of Public Roads (BPR) congestion function is used \cite{bstabler}:
\begin{gather*}         
         \tau_e(f_e) = \bar{t}_e \left(1 + \rho \left( \frac{f_e}{\bar{f}_e}\right)^{\frac{1}{\mu}} \right) .
         \hspace{1cm} \tau(f) \equiv \{\tau_e(f_e)\}_{e \in \mathcal{E}} ,
\end{gather*}
where $\rho , \mu , \overline{t_e} , \overline{f_e}$ are road parameters in the model.

The model is given a predefined set of source and destinations pairs $w = (i,j)$ and the amount of traffic (correspondence) $d_w$ between each origin-destination pair in the considered period of time. 
\begin{gather*}
    \hspace{1cm} O \text{ and } D -\text{sets of origin and destination vertices } \subset \mathcal V \\
    \hspace{1cm} OD \equiv \{ w = (i,j) ~|~ i \in O , j \in D \} \\
    \hspace{1cm} \mathcal{P} \equiv \{\mathcal{P}_{w}\}_{w\in OD},~\mathcal{P}_{w} \text{ is the set of all paths from $i$ to $j$ for $w = (i,j)$}  
\end{gather*}

The correspondences of each origin-destination pair are additively partitioned into flows along the paths connecting these pairs:
\begin{gather*}
    \hspace{1cm} \mathcal X = \{ x \; | \; d_w = \sum\limits_{p\in \mathcal{P}_w} x_p \; , \; x_p \ge 0 \; \forall w \in OD \}. 
\end{gather*}

Flows on all paths along some road overlap and give rise to some value of the flow on that road.  
The set of feasible edge flows is given by the image of \( \mathcal{X} \) under the linear operator defined by the path-edge incidence matrix \( \Theta \in \{0,1\}^{|\mathcal E| \times |\mathcal{P}|} \):
\[
\mathcal{F} = \left\{ f \in \mathbb{R}^{|\mathcal E|} \,\middle|\, \exists\, x \in \mathcal{X} \text{ such that } f = \Theta x \right\}.
\]

\subsection{Frank-Wolfe}

\begin{algorithm}[H]
    \begin{algorithmic}[1]
        \STATE $f_0 \in \mathcal F$~--- starting point 
        \STATE $k := 0$
        \REPEAT
            \STATE ${s_k^{FW}} := \arg\min\limits_{s \in \mathcal{F}} \langle \nabla \Psi(f_k) , s \rangle$\label{line:linmin}
            \STATE $d_k := {s_k^{FW}} - f_k$\label{line:stepdir}
            \STATE $\gamma_k := \arg\min\limits_{\gamma \in [0,1]} \Psi( f_k + \gamma \cdot d_k )$ \label{line:linesearch} 
            \STATE $f_{k + 1} := f_k + \gamma_k \cdot d_k$ \label{line:step}
            \STATE $k := k + 1$
        \UNTIL{ $k < \text{max iter}$}
    \end{algorithmic}
    \caption{Frank--Wolfe}
    \label{alg::frank-wolfe}
\end{algorithm}

The main loop of the algorithm begins by applying the linear minimization oracle in line~\ref{line:linmin}.
Next, the step direction is obtained in line~\ref{line:stepdir}. 
In lines~\ref{line:linesearch}--\ref{line:step} the stepsize  is searched using linesearch and the step is performed.

Since $\nabla \Psi(f_k) = \tau(f_k)$, the linear minimization oracle minimizes total travel time experienced by all users in each correspondence: 
\begin{gather*}
   \min\limits_{s \in \mathcal F} \langle \nabla \Psi(f^k) , s \rangle = \min\limits_{s \in \mathcal F} \langle \tau(f_k) , s \rangle = \sum\limits_{w \in OD } T_{p_w^*}(f_k) \cdot d_w,
\end{gather*}
where $p_w^*$ is one of the shortests paths between origin $i$ and destination $d$ of the correspondence $w = (i,j) \in OD$,  $T_{p_w^*}(f_k)$ -- travel time along the shortest path $p_w^*$.
Therefore, the solution to this linear subproblem is to distribute all traffic along the shortest paths.
Variable $s \in \mathcal F$ is then efficiently calculated from the output of Dijkstra's or similar one-to-all shortest paths algorithm by summing flows from the shortest paths trees for all origins.

The key observation is that for the Frank-Wolfe step it is necessary to find  shortest paths over for correspondences. On large road graphs this can become computationally expensive. The main interest is in how to improve the method, by making single iteration more computationally efficient.

Note, that the whole algorithm can also be formulated and implemented in the path-based fashion, i.e., in the space of flows on paths $\mathcal X$.
The present link-based form with flows on edges is preferable due to lower memory usage: in the worst case at each iteration a new shortest path will be found, thus the path-based implementation will require to store hundreds of paths for each correspondence.
However, we will use path-based notation in the next section to make the reduction to the block-coordinate Frank-Wolfe algorithm. 
Despite that, in the actual implementation we instead use an origin-based representation to store intermediate results.

\newpage

{\centering
\section{Stochastic Origin Frank-Wolfe}
}

The proposed \textbf{Stochastic Origin Frank-Wolfe} (SOFW) method is a structured variant of the Block-Coordinate Frank-Wolfe algorithm~\cite{jaggi2012blockcoordfw}, adapted to the setting of equilibrium distribution flows in transportation networks.

We consider the path flow vector \( x \in \mathcal{X} \subseteq \mathbb{R}^{|\mathcal{P}|} \) as a concatenation of blocks:
\[
x = \big( x^{(i,j)} \big)_{(i,j) \in \text{OD}}, \quad x^{(i,j)} \in \mathbb{R}^{|\mathcal{P}_{i,j}|},
\]
where each block corresponds to the set of paths \( \mathcal{P}_{i,j} \) between a given origin-destination (OD) pair \((i,j)\). These blocks satisfy the demand constraints:
\[
\sum_{p \in \mathcal{P}_{i,j}} x_p = d_{ij}, \quad x_p \geq 0.
\]

At each iteration, SOFW samples a random subset of OD pairs \( \mathcal{I} \subset \text{OD} \), and updates only the corresponding sub-blocks \( x^{(i,j)} \), while freezing the rest. The optimization is thus restricted to the subspace:
\[
\mathcal{X}_{\mathcal{I}} = \left\{ x \in \mathcal{X} \,\middle|\, x^{(i,j)} = \bar{x}^{(i,j)} \text{ for } (i,j) \notin \mathcal{I} \right\},
\]
and the linear oracle is solved over the set \( F(\mathcal{X}_{\mathcal{I}}) = \{ \Theta x \mid x \in \mathcal{X}_{\mathcal{I}} \} \subset \mathcal F\).

For rational application of shortest-paths algorithms we sample a random subset by sampling a fraction of origin vertices and including all corresponding OD pairs.
This results in a computationally efficient approximation of the classical Frank-Wolfe direction that require only several calls to the Dijkstra's or similar shortest-paths algorithm to be computed, while preserving convergence guarantees established in the block-coordinate framework~\cite{jaggi2012blockcoordfw}.
The complete method is presented in Algorithm~\ref{alg::stoch-frank-wolfe}.

\begin{algorithm}[H]
    \caption{Stochastic Origin Frank-Wolfe}
    \label{alg::stoch-frank-wolfe}
    \begin{algorithmic}[1]
        \STATE Choose initial flow \( f_0 \in \mathcal{F} \), distribution $P((i,j))= w_{i,j}$ 
        \STATE Set iteration counter \( k := 0 \)
        \REPEAT
            \STATE Sample a random subset \( \mathcal{I}_k \subset \text{OD} \)  with probabilities \( \propto w_{ij} \)\hfill \textit{(selected OD-pairs)}
            \STATE Define subspace \( \mathcal{X}_k = \{x \in \mathcal{X} \mid x^{(i,j)} = x_k^{(i,j)} \text{ for all } (i,j) \notin \mathcal{I}_k \} \)
            \STATE $s_k := \arg\min_{s \in F(\mathcal{X}_k)} \langle \nabla \Psi(f_k), s \rangle$
            
            \STATE Compute direction: 
                \[
                d_k := \frac{1}{| \mathcal{I}_k |}W_k \cdot \sum_{(i,j) \in \mathcal{I}_k} \frac{1}{w_{ij}} \left( F(s_k^{[i,j]}) - F(x_k^{[i,j]}) \right), \quad
                W_k := \left( \sum_{(i,j) \in \mathcal{I}_k} \frac{1}{w_{ij}} \right)^{-1}
                \]
            \STATE Choose step size \( \gamma_k \in [0,1] \), e.g., \( \gamma_k = \frac{2}{k+1} \) or via line search
            \STATE Update flow: \( f_{k+1} := f_k + \gamma_k d_k \)
            \STATE Increment \( k := k + 1 \)
        \UNTIL{stopping criterion is satisfied}
    \end{algorithmic}
\end{algorithm}

In the weighted variant, origin-destination (OD) pairs are sampled by first selecting sources with probabilities proportional to their total demand, and then uniformly sampling correspondences associated with the chosen sources. Formally, the probability of sampling a correspondence $(i,j)$ is given by

$$
\mathbb{P}((i,j)) = \frac{d_i}{\sum_{k} d_k} \cdot \frac{1}{|\mathcal{C}_i|},
$$

where $d_i = \sum_{j} d_{ij}$ is the total demand of source $i$, and $|\mathcal{C}_i|$ is the number of correspondences originating from source $i$. This approach ensures that sources with higher total demand are sampled more frequently, while correspondences within each source are equally weighted. An appropriate importance-weighted correction is applied to guarantee an unbiased estimation of the Frank-Wolfe gap.
For simplicity, in our experiments we do not sample individual correspondences directly. Instead, we perform batched sampling over sources: at each iteration, we randomly (with weights $\sim \frac{d_i}{\sum_{k} d_k}$) select a fraction $\alpha$ of all sources. Then, for each selected source $i$, we include all of its associated correspondences $(i,j)$ without additional subsampling. This approach is natural and efficient, as it allows computing shortest paths from a given source to all destinations in a single step, thereby reducing overhead while preserving the sparsity and scalability benefits of the SOFW method.

{\centering
\section{Method convergence}
}

The convergence of the proposed SOFW method follows from the general analysis of the Block-Coordinate Frank-Wolfe algorithm on product domains~\cite{jaggi2012blockcoordfw}. We consider a non-batch setting, where at each iteration the algorithm selects a source uniformly and randomly and selects all blocks corresponding to it. This corresponds to the unweighted variant of SOFW with the sample size equal to the number of destinations, i.e., a single source is selected and all corresponding OD pairs originating from it are updated simultaneously. While the current analysis focuses on this regime, it can be extended to batched updates, where several blocks are sampled and updated simultaneously, though this would require a more general theoretical treatment.

In our setting, the feasible set of path flows is a Cartesian product:
\[
\mathcal{X} = \prod_{(i,j) \in \text{OD}} \mathcal{X}^{(i,j)},
\]
and the set of feasible edge flows is given by the linear mapping \( F : \mathcal{X} \to \mathcal F \), defined by \( F(x) = \Theta x \).

\paragraph{Block curvature constant.}
Let \( \Psi : \mathbb{R}^{|E|} \to \mathbb{R} \) be a convex and differentiable objective (e.g., the Beckmann potential). The curvature of \( \Psi \circ F\) over block \( (i,j) \in \text{OD} \) is defined as:
\[
C^{(i,j)}_\Psi := \sup_{\substack{x \in \mathcal{X},\ s^{(i,j)} \in \mathcal{X}^{(i,j)},\\ \gamma \in [0,1],\ y = x + \gamma (s^{[i,j]} - x^{[i,j]})}} \frac{2}{\gamma^2} \left[ \Psi(F(y)) - \Psi(F(x)) - \left\langle \nabla \Psi(F(x)),\ F(y^{[i,j]}) - F(x^{[i,j]}) \right\rangle \right],
\]
where $s^{[i,j]}, x^{[i.j]} \in \mathcal{X}$ denotes a one-hot vectors with nonzero entries only in the components indexed by $\mathcal{P}_{i,j}$.

The total curvature over the product domain is then given by:
\[
C^{\otimes}_\Psi := \sum_{(i,j) \in \text{OD}} C^{(i,j)}_\Psi.
\]

\paragraph{Primal convergence.}
Consider the variant of SOFW where at each iteration a single correspondence $(i,j) \in \mathrm{OD}$ is sampled uniformly at random and updated. 
Let \( f_k = \mathcal{F}(x_k) \) denote the iterate at step \( k \). 
Then, for this exact variant, the expected suboptimality satisfies~\cite{jaggi2012blockcoordfw}:
\[
\mathbb{E}[\Psi(f_k)] - \Psi(f^*) 
\;\le\; 
\frac{2|\mathrm{OD}|}{k + 2|\mathrm{OD}|} 
\left[ C^{\otimes}_\Psi + \Psi(f_0) - \Psi(f^*) \right].
\]

Therefore, to obtain an \( \varepsilon \)-accurate solution in expectation, it suffices to run
\[
\mathcal{O}\!\left( \frac{|\mathrm{OD}|}{\varepsilon} \right)
\]
iterations.

This result can be extended to the batched setting, where multiple blocks are sampled and updated simultaneously, although the theoretical rate in terms of the iteration count remains of the same order.

\paragraph{Batched primal convergence.}

We consider the batched SOFW method that, at iteration $k$, samples (without replacement) a mini-batch $\mathcal{I}_k \subset \mathrm{OD}$ of size $m$ uniformly at random, solves the usual Frank--Wolfe (Frank-Wolfe) linear subproblem on each sampled block, and moves along the averaged direction.

\begin{align*}
d_k 
&= \frac{1}{m} \sum_{(i,j)\in\mathcal{I}_k} \bigl( s_k^{[i,j]} - x_k^{[i,j]} \bigr),\\
x_{k+1} 
&= x_k + \gamma_k d_k,\qquad \gamma_k \in [0,1].
\end{align*}

Let the full Frank--Wolfe gap decompose across blocks as
\begin{align*}
g(x) 
&:= \sum_{(i,j)\in\mathrm{OD}} g^{(i,j)}(x),\\
g^{(i,j)}(x) 
&:= \Big\langle \nabla \Psi(F(x)),\, F\!\bigl(x^{[i,j]} - s^{[i,j]}(x)\bigr) \Big\rangle,
\end{align*}
where $s^{[i,j]}(x)$ denotes an optimal Frank--Wolfe atom for block $(i,j)$ at $x$ (lifted to the product space).

Recall the block curvature constants $C_{\Psi}^{(i,j)}$ and their sum
\begin{align*}
C_{\Psi}^{\otimes} := \sum_{(i,j)\in\mathrm{OD}} C_{\Psi}^{(i,j)}.
\end{align*}

We first record two auxiliary facts.

\begin{remark}[Single-block progress inequality]
    For any $x\in\mathcal{X}$, block $(i,j)\in\mathrm{OD}$, atom $s^{[i,j]}(x)$, and step $\gamma\in[0,1]$, by the definition of the block curvature constant $C_{\Psi}^{(i,j)}$, we have
    \begin{align*}
    \Psi\!\bigl(F\!\bigl(x + \gamma (s^{[i,j]}(x)-x^{[i,j]})\bigr)\bigr)
    \;\le\;
    \Psi(F(x))
    - \gamma\, g^{(i,j)}(x)
    + \frac{\gamma^2}{2}\, C_{\Psi}^{(i,j)}.
    \end{align*}
\end{remark}

\begin{lemma}[Convex averaging across a mini-batch]
    Let $\{v_r\}_{r=1}^m$ be points in $\mathcal{X}$ and $v = \frac{1}{m}\sum_{r=1}^m v_r$. By convexity of $\Psi\circ F$,
    \begin{align*}
        \Psi(F(v)) \le \frac{1}{m} \sum_{r=1}^m \Psi(F(v_r)).
    \end{align*}
\end{lemma}

Applying this with $v_r = x_k + \gamma_k(s_k^{[i_r,j_r]} - x_k^{[i_r,j_r]})$ for $(i_r,j_r)\in\mathcal{I}_k$ and $v = x_k + \gamma_k d_k$ gives
\begin{align*}
\Psi(F(x_{k+1})) 
&= \Psi\!\Bigl(F\!\Bigl(x_k + \gamma_k \frac{1}{m}\sum_{(i,j)\in\mathcal{I}_k} (s_k^{[i,j]} - x_k^{[i,j]})\Bigr)\Bigr) \\
&\le \frac{1}{m} \sum_{(i,j)\in\mathcal{I}_k} 
\Psi\!\bigl(F(x_k + \gamma_k (s_k^{[i,j]} - x_k^{[i,j]}))\bigr).
\end{align*}

Invoking the single-block progress inequality term-wise gives, for any realized mini-batch $\mathcal{I}_k$,
\begin{align*}
\Psi(F(x_{k+1}))
&\le
\Psi(F(x_k))
- \gamma_k \frac{1}{m} \sum_{(i,j)\in\mathcal{I}_k} g^{(i,j)}(x_k)
+ \frac{\gamma_k^2}{2} \frac{1}{m} \sum_{(i,j)\in\mathcal{I}_k} C_{\Psi}^{(i,j)}.
\end{align*}

Taking the conditional expectation w.r.t.\ the random mini-batch $\mathcal{I}_k$ (uniform sampling without replacement), we have
\begin{align*}
\mathbb{E}\!\bigl[\Psi(F(x_{k+1})) \,\big|\, x_k\bigr]
&\le
\Psi(F(x_k))
- \gamma_k \frac{1}{m} \,\mathbb{E}\!\Bigl[\sum_{(i,j)\in\mathcal{I}_k} g^{(i,j)}(x_k)\Bigr]
+ \frac{\gamma_k^2}{2} \frac{1}{m} \,\mathbb{E}\!\Bigl[\sum_{(i,j)\in\mathcal{I}_k} C_{\Psi}^{(i,j)}\Bigr] \\
&=
\Psi(F(x_k))
- \gamma_k \frac{1}{m} \cdot m \cdot \frac{1}{|\mathrm{OD}|}\sum_{(i,j)\in\mathrm{OD}} g^{(i,j)}(x_k)
+ \frac{\gamma_k^2}{2} \frac{1}{m} \cdot m \cdot \frac{1}{|\mathrm{OD}|}\sum_{(i,j)\in\mathrm{OD}} C_{\Psi}^{(i,j)} \\
&=
\Psi(F(x_k))
- \gamma_k \frac{1}{|\mathrm{OD}|}\, g(x_k)
+ \frac{\gamma_k^2}{2}\,\frac{1}{|\mathrm{OD}|}\, C_{\Psi}^{\otimes}.
\end{align*}

\medskip
\textbf{Observation:} the expected update recursion is \emph{identical} to the non-batched case. Thus, batching does not change the theoretical convergence rate in terms of iteration count, but reduces the variance and may improve wall-clock time in practice.

\begin{theorem}[Expected primal convergence of batched SOFW]
    \label{thm:batched-SOFW}
    Let $\{x_k\}_{k\ge 0}$ be the iterates of the batched SOFW method defined above with any batch size $m$ and step sizes
    \begin{align*}
        \gamma_k = \frac{2}{k + 2|\mathrm{OD}|}.
    \end{align*}
    Then, for all $k \ge 0$,
    \begin{align*}
        \mathbb{E}\bigl[\Psi(F(x_k))\bigr] - \Psi(F(x^*))
        \;\le\;
        \frac{2|\mathrm{OD}|}{k + 2|\mathrm{OD}|}
        \Bigl[ C_{\Psi}^{\otimes} + \Psi(F(x_0)) - \Psi(F(x^*)) \Bigr].
    \end{align*}
    Consequently, an $\varepsilon$-accurate solution in expectation is obtained after
    \begin{align*}
        \mathcal{O}\!\left( \frac{|\mathrm{OD}|}{\,\varepsilon} \right)
    \end{align*}
    iterations.
\end{theorem}

\begin{proof}
    Define the suboptimality gap 
    \(
        \Delta_k := \mathbb{E}[\Psi(F(x_k))] - \Psi(F(x^*)) .
    \) 
    From the above conditional expectation, we have
    \begin{align*}
        \Delta_{k+1} 
        &\le 
        \Delta_k 
        - \gamma_k \frac{1}{|\mathrm{OD}|}\, \mathbb{E}[g(x_k)]
        + \frac{\gamma_k^2}{2}\,\frac{1}{|\mathrm{OD}|}\, C_{\Psi}^{\otimes}.
    \end{align*}
    Since \( \mathbb{E}[g(x_k)] \ge \Delta_k \) (the Frank--Wolfe gap dominates the primal gap in expectation), we obtain
    \begin{align*}
        \Delta_{k+1} 
        &\le 
        \Delta_k \left( 1 - \frac{\gamma_k}{|\mathrm{OD}|} \right)
        + \frac{\gamma_k^2}{2|\mathrm{OD}|} C_{\Psi}^{\otimes}.
    \end{align*}
    Choosing \( \gamma_k = \tfrac{2}{k + 2|\mathrm{OD}|} \), the standard induction argument used in the proof of Theorem C.2 of \cite{jaggi2012blockcoordfw} (with $\nu = 1$) directly applies here and yields
    \begin{align*}
        \Delta_k 
        \le 
        \frac{2|\mathrm{OD}|}{k + 2|\mathrm{OD}|}
        \Bigl[ C_{\Psi}^{\otimes} + \Psi(F(x_0)) - \Psi(F(x^*)) \Bigr].
    \end{align*}
    This concludes the proof.
\end{proof}

\paragraph{Remarks on batching.}
The above analysis extends the standard Block-Coordinate Frank--Wolfe framework \cite{jaggi2012blockcoordfw} 
to the batched setting, which was not explicitly covered in the original paper. 
In our experiments, batching arises naturally because we sample all OD-pairs corresponding 
to one or several sources at each iteration.

It is important to note that batching, as analyzed here, does not improve the theoretical iteration complexity: 
the expected progress per iteration remains the same due to the normalization by the batch size \(m\).
To achieve an improved rate through batching, one would need to adjust certain assumptions or 
redefine curvature constants (e.g., taking advantage of block-wise structure or stronger smoothness). 
Such refinements are beyond the current scope and would require a more careful theoretical treatment.

\paragraph{Remarks on stochastic dual methods for traffic assignment.}
We revealed two difficulties in applying stochastic primal-dual methods in this setup.
First, there is no ready-to-use variant of USTM (which is the only known primal-dual method whose performance can compare with Frank-Wolfe) with inexact function value oracle \cite{kubentayeva2021finding,kubentayeva2024primal}: by default universal methods require to compare values of the objective function in the current and next candidate point \cite{gasnikov2018universal}, what require finding all shortest paths.
Second, the convergence guarantees for primal-dual methods only hold then the primal variable is restored using special averaging formulas, which, in case of stochastic steps, yield non-feasible approximate solutions (i.e. flows do not satisfy the demand constraints), what makes them far less favorable for practical applications.

{\centering
\section{Experiments}
}

We evaluate the performance of the proposed \emph{Stochastic Origin Frank-Wolfe} (SOFW) method against the classical Frank-Wolfe (FW) algorithm, as well as several variants involving different sampling strategies of origin-destination correspondences. In particular, we compare the standard SOFW variant that uniformly samples a subset of correspondences, with varying fractions $\alpha$ of the total number of origin-destination pairs $|OD|$, and the weighted variant (SOFW-w) described in the algorithm section, which samples correspondences proportionally to the weights of their respective sources — where each source weight equals the sum of demands over all correspondences originating from it. For both variants, each correspondence linked to a given source is assigned an equal weight. We also include the classical FW method as a baseline. Experiments are conducted on large-scale transportation datasets from the \texttt{TransportationNetworks} repository: \texttt{Chicago-Regional}, \texttt{GoldCoast}, \texttt{Birmingham-England}, \texttt{Chicago-Sketch}, and \texttt{Philadelphia}. Additionally, we evaluate on smaller benchmark instances: \texttt{SiouxFalls}, \texttt{Anaheim}, \texttt{Terrassa-Asymmetric}, and \texttt{Winnipeg}.

The SOFW algorithm is designed as follows: suppose there are $n$ origin nodes and $m$ destination nodes (i.e., $n \times m$ OD-pairs). At each iteration, we sample a fixed fraction of sources—e.g., $10\%$ of the total origins—and compute shortest paths from these selected sources to all destinations. This reduces the computational cost of shortest-path search by roughly an order of magnitude. 

To implement this scheme, we decompose the total flow into components corresponding to each origin node. At each iteration, only the flows associated with a sampled subset of origins \(\mathcal{A} \subseteq \{1, \dots, n\}\) are updated, while flows corresponding to the remaining origins are fixed at their current values. Formally, the total edge flow \(f\) can be expressed as
\[
    f = \sum_{a=1}^n f^{(a)},
\]
where \(f^{(a)}\) denotes the flow induced by origin \(a\). During the algorithm, we update only the components \(f^{(a)}\) for \(a \in \mathcal{A}\), fixing the others. 

Although maintaining separate flow components increases memory usage by a factor of approximately
\[
    \frac{n}{|\mathcal{A}|},
\]
where \(n\) is the total number of origins and \(|\mathcal{A}|\) is the number of sampled origins per iteration, this approach significantly reduces per-iteration computational cost by limiting shortest path computations to the sampled subset.

Surprisingly, this memory-compute tradeoff leads to substantial practical improvements. Despite increased memory usage, SOFW achieves nearly an order-of-magnitude better Frank-Wolfe gap for the same computational budget compared to the classical FW method.

An important factor behind the growing performance gap between SOFW and the classical FW method on larger datasets lies in the relationship between the stochastic update quality and the computational complexity of shortest-path calculations.

Empirically, the quality of the stochastic update remains relatively stable even when only a small fraction of origins is selected—i.e., the accuracy of the descent direction does not degrade in proportion to the sampling rate. In contrast, the computational complexity of a full Frank-Wolfe step scales as

$$
    \mathcal{O}(n \cdot (E + V \log V)),
$$

where $n$ is the number of origin nodes, $V = |\mathcal{V}|$ is the number of nodes in the network, and $E = |\mathcal{E}|$ is the number of edges. This reflects the cost of solving single-source shortest path problems from all origins.

By sampling only a fraction $\alpha$ of the origins (e.g., $\alpha = 0.1$), SOFW reduces the per-iteration cost to $\mathcal{O}(\alpha \cdot n \cdot (E + V \log V)),$ leading to substantial savings in large-scale settings.

As the number of origins and the network size increase, this computational advantage becomes more pronounced, resulting in faster convergence and better practical performance compared to the classical FW algorithm.

We summarize key experimental findings:
\begin{enumerate}
    \item SOFW significantly outperforms FW on large-scale graphs.
    \item The advantage of SOFW increases with the size of the network.
    \item On small networks, SOFW still performs competitively, although the gap between FW and SOFW becomes less pronounced due to the overhead from flow decomposition.
    \item A clear memory-quality tradeoff is observed: increasing memory usage via flow decomposition yields better convergence rates.
    \item The weighted variant SOFW-w demonstrates faster initial progress compared to standard SOFW, but the performance of both methods converges in the long run.
\end{enumerate}

\vspace{-1cm}
\begin{figure}[H]
    \begin{minipage}{0.5\textwidth}
        \centering
        \includegraphics[width=\linewidth]{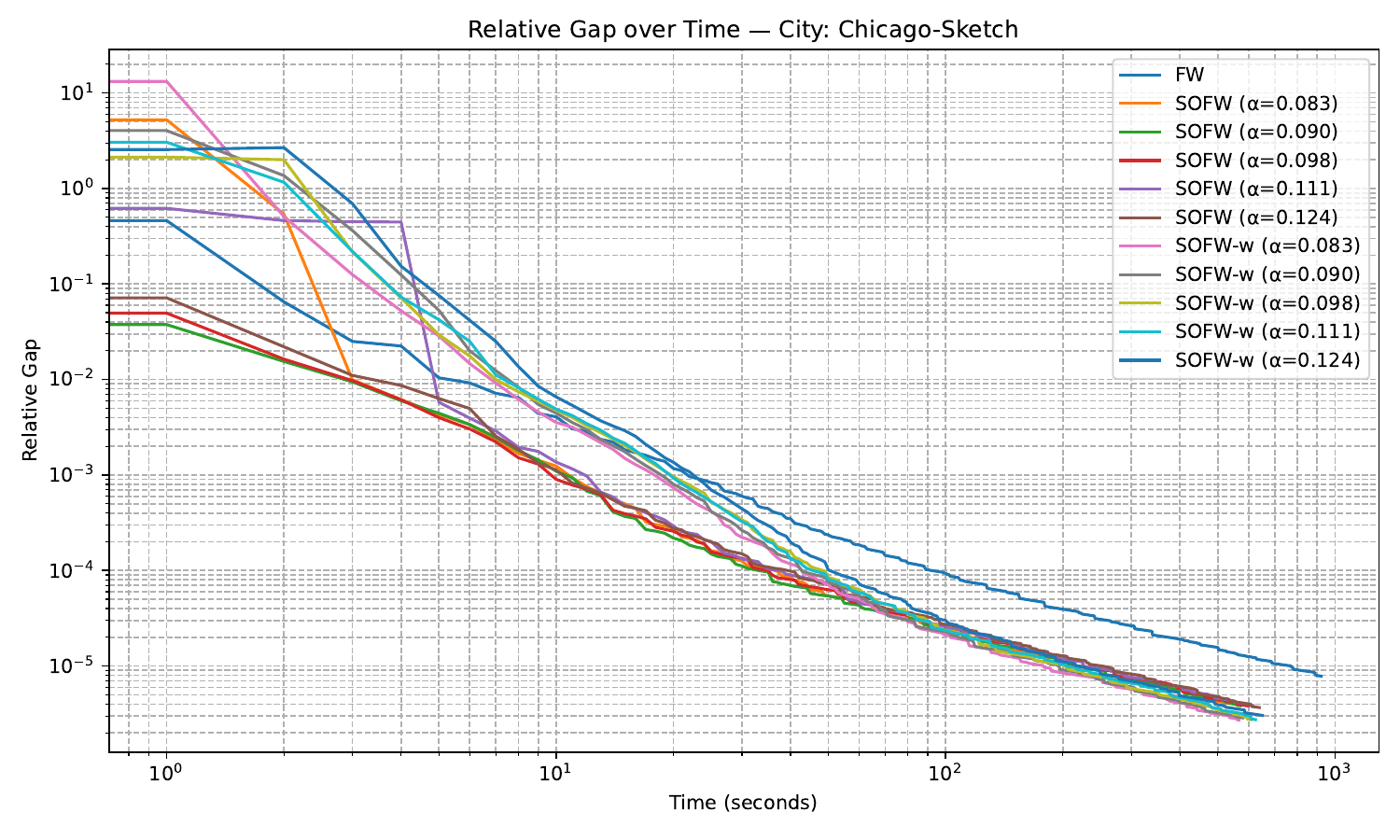}
        \caption{Dataset: Chicago-Sketch}
    \end{minipage}%
    \begin{minipage}{0.5\textwidth}
        \centering
        \includegraphics[width=\linewidth]{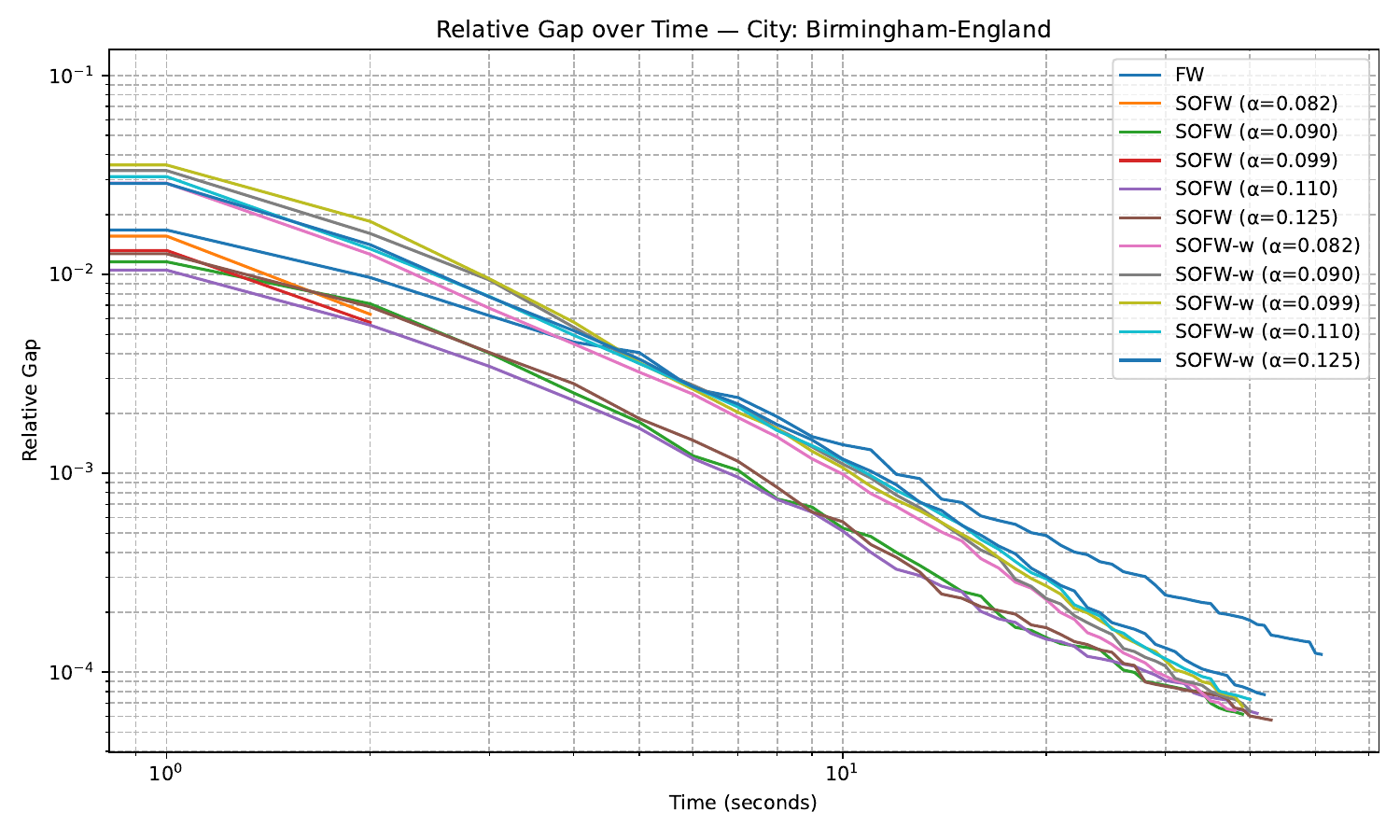}
        \caption{Dataset: Birmingham-England.}
    \end{minipage}
    \caption{Comparison of SOFW and FW on low-scale datasets.}
    \label{fig:SOFW1}
\end{figure}
\vspace{-1cm}
\begin{figure}[H]
    \begin{minipage}{0.5\textwidth}
        \centering
        \includegraphics[width=\linewidth]{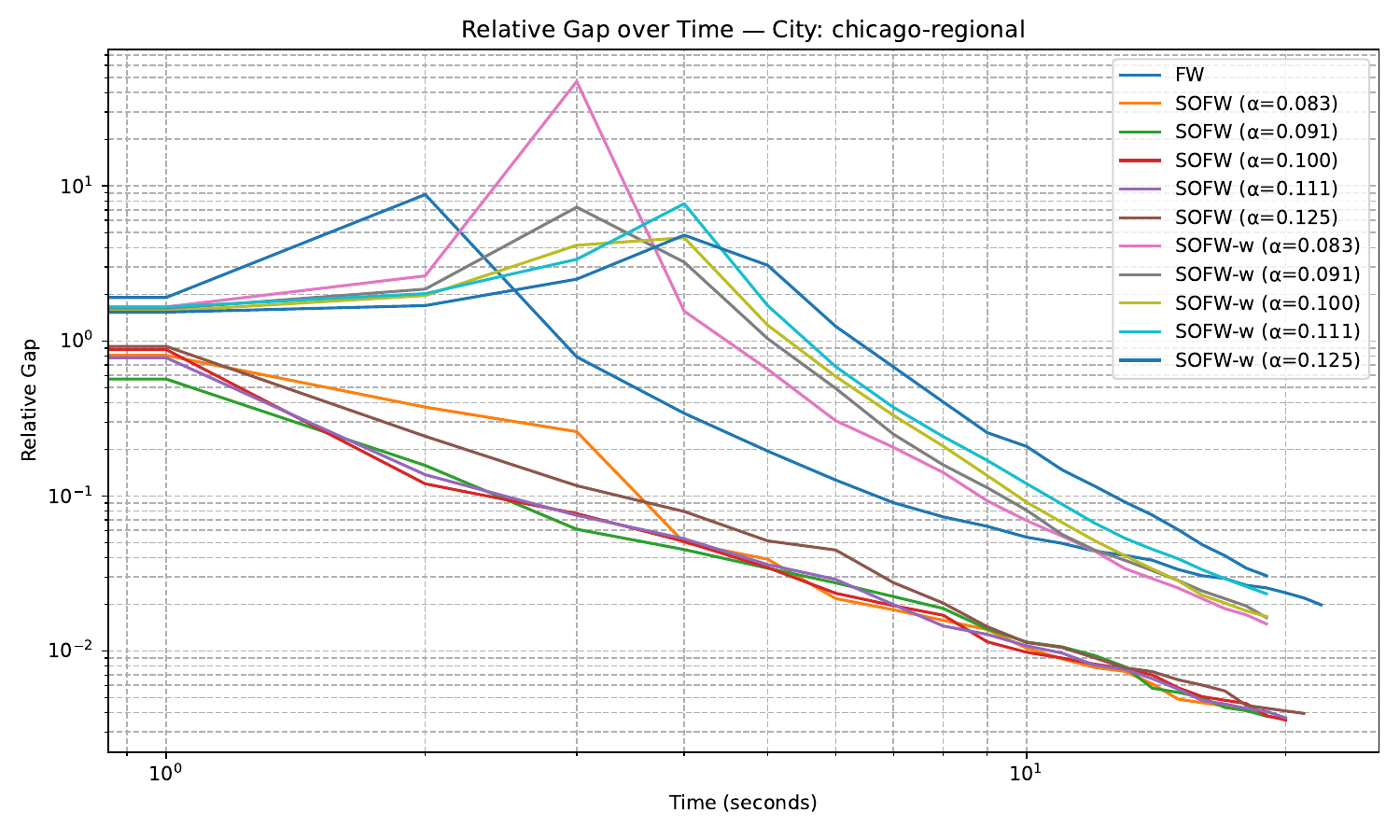}
        \caption{Dataset: Chicago-Regional.}
    \end{minipage}%
    \begin{minipage}{0.5\textwidth}
        \centering
        \includegraphics[width=\linewidth]{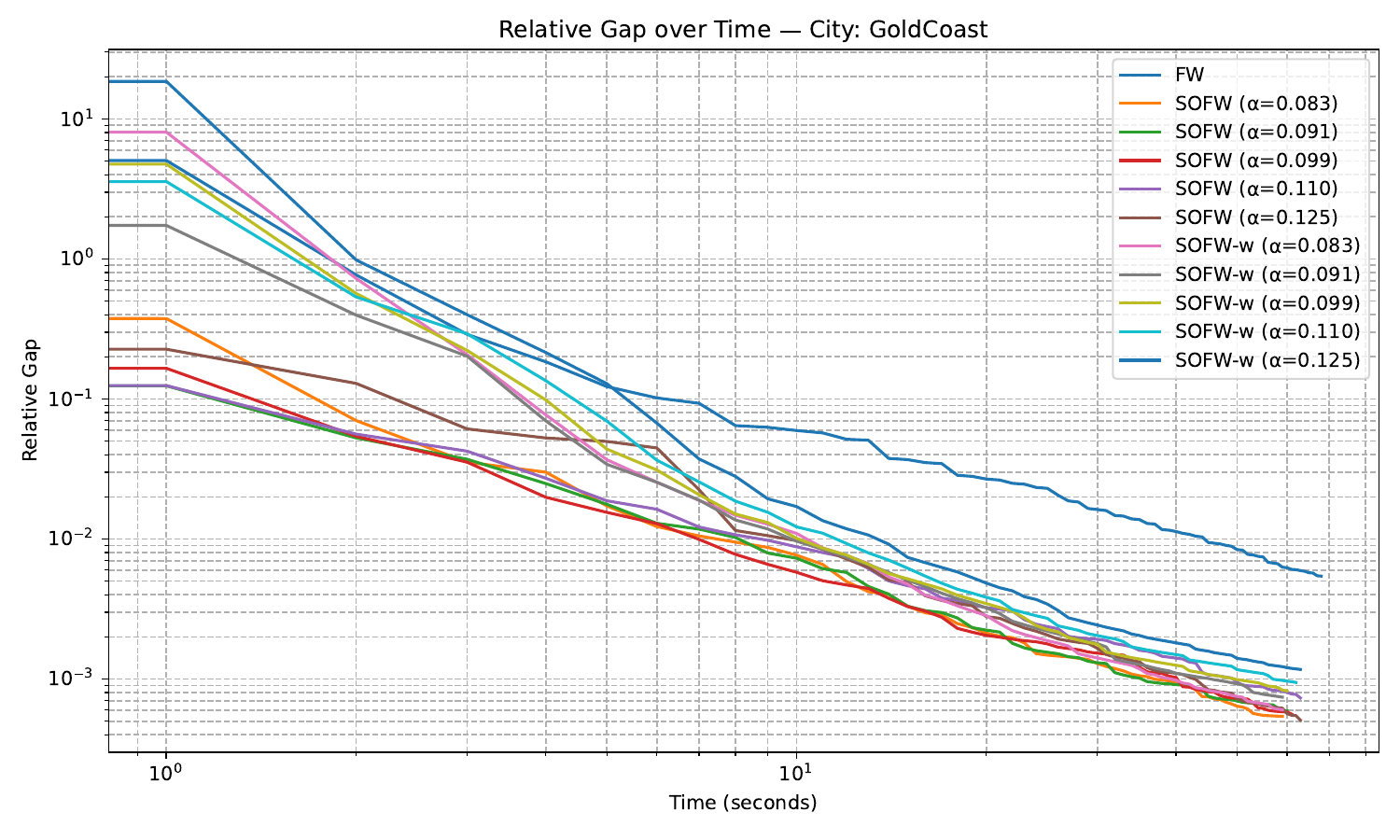}
        \caption{Dataset: GoldCoast}
    \end{minipage}
    \caption{SOFW effectiveness on medium-size transportation networks.}
    \label{fig:SOFW2}
\end{figure}
\vspace{-0.2cm}
\begin{figure}[H]
    \centering
    \includegraphics[width=0.8\textwidth]{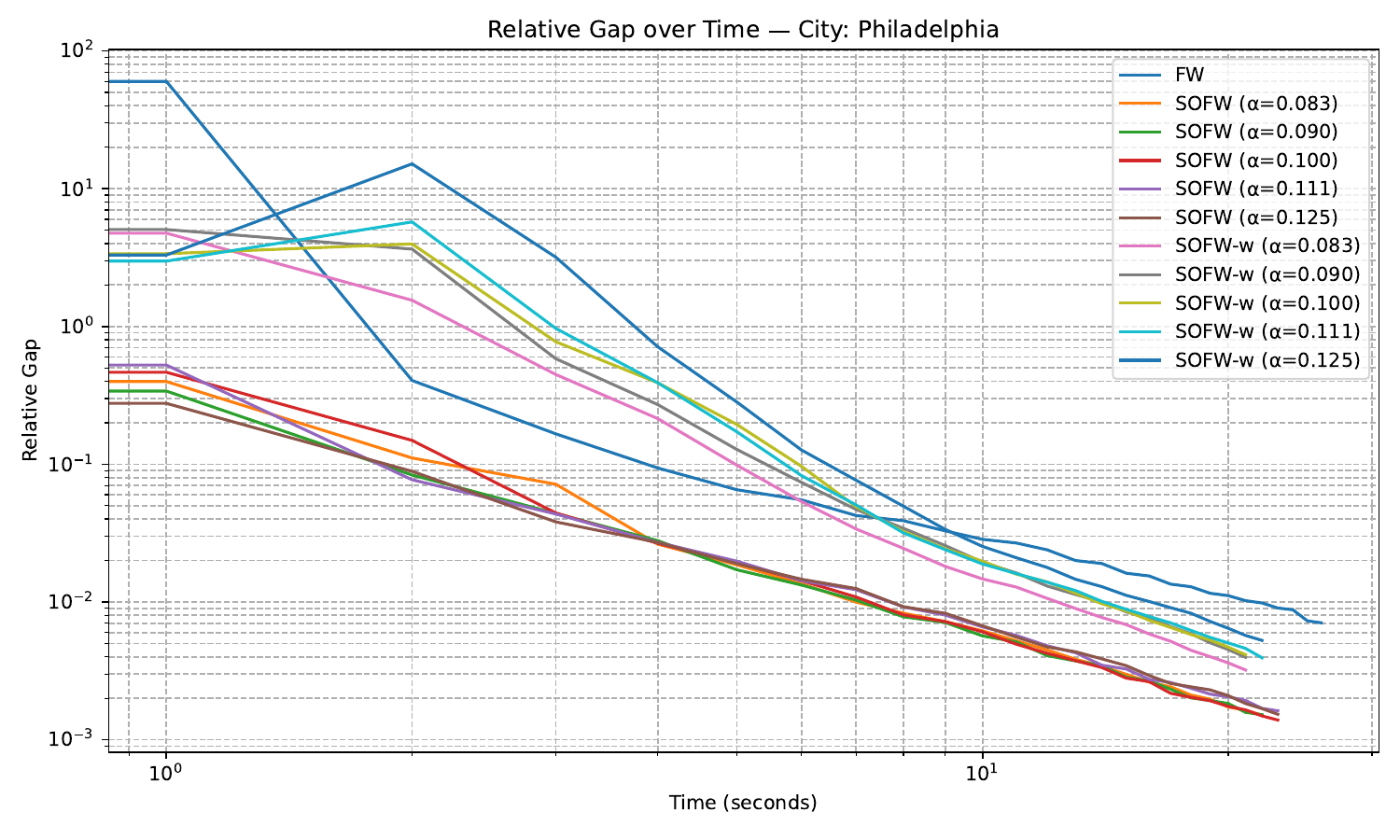}
    \caption{Dataset: Philadelphia. Comparison of SOFW and FW on large-scale dataset.}
    \label{fig:SOFW3}
\end{figure}
\vspace{-1cm}
\begin{figure}[H]
    \begin{minipage}{0.5\textwidth}
        \centering
        \includegraphics[width=\linewidth]{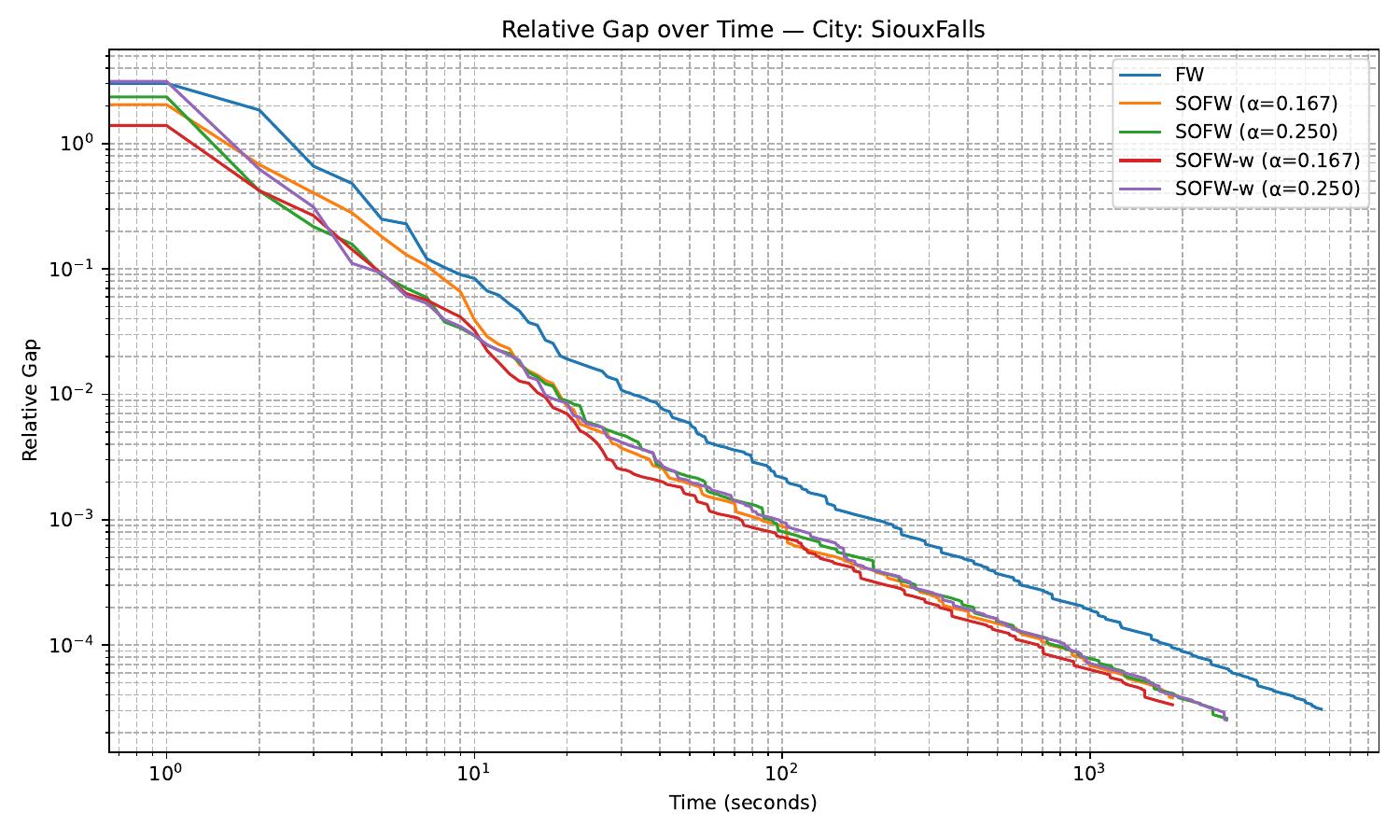}
        \caption{Dataset: SiouxFalls. SOFW vs FW on small network.}
    \end{minipage}%
    \begin{minipage}{0.5\textwidth}
        \centering
        \includegraphics[width=\linewidth]{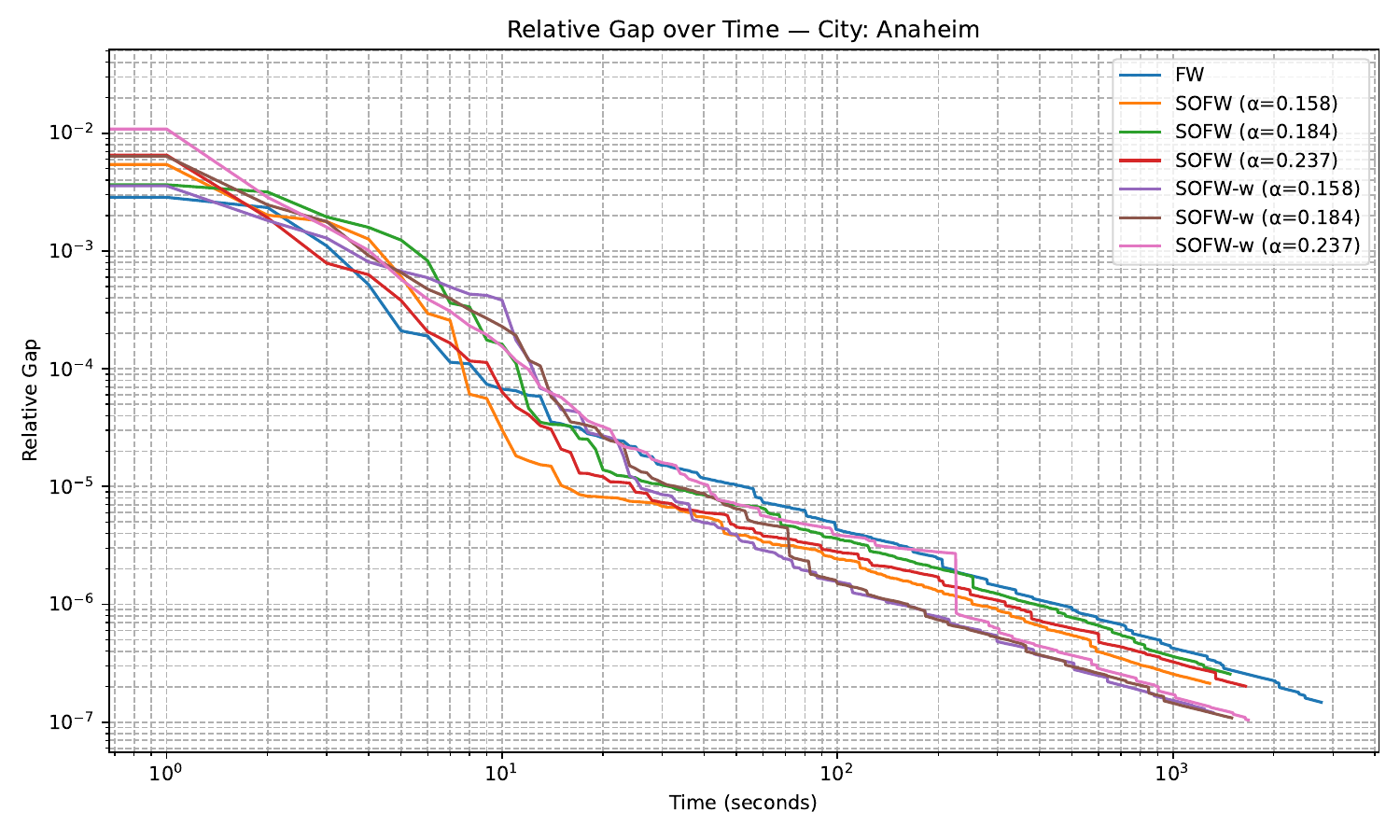}
        \caption{Dataset: Anaheim. SOFW vs FW on small network.}
    \end{minipage}
    \vspace{0.3cm}

    \begin{minipage}{0.5\textwidth}
        \centering
        \includegraphics[width=\linewidth]{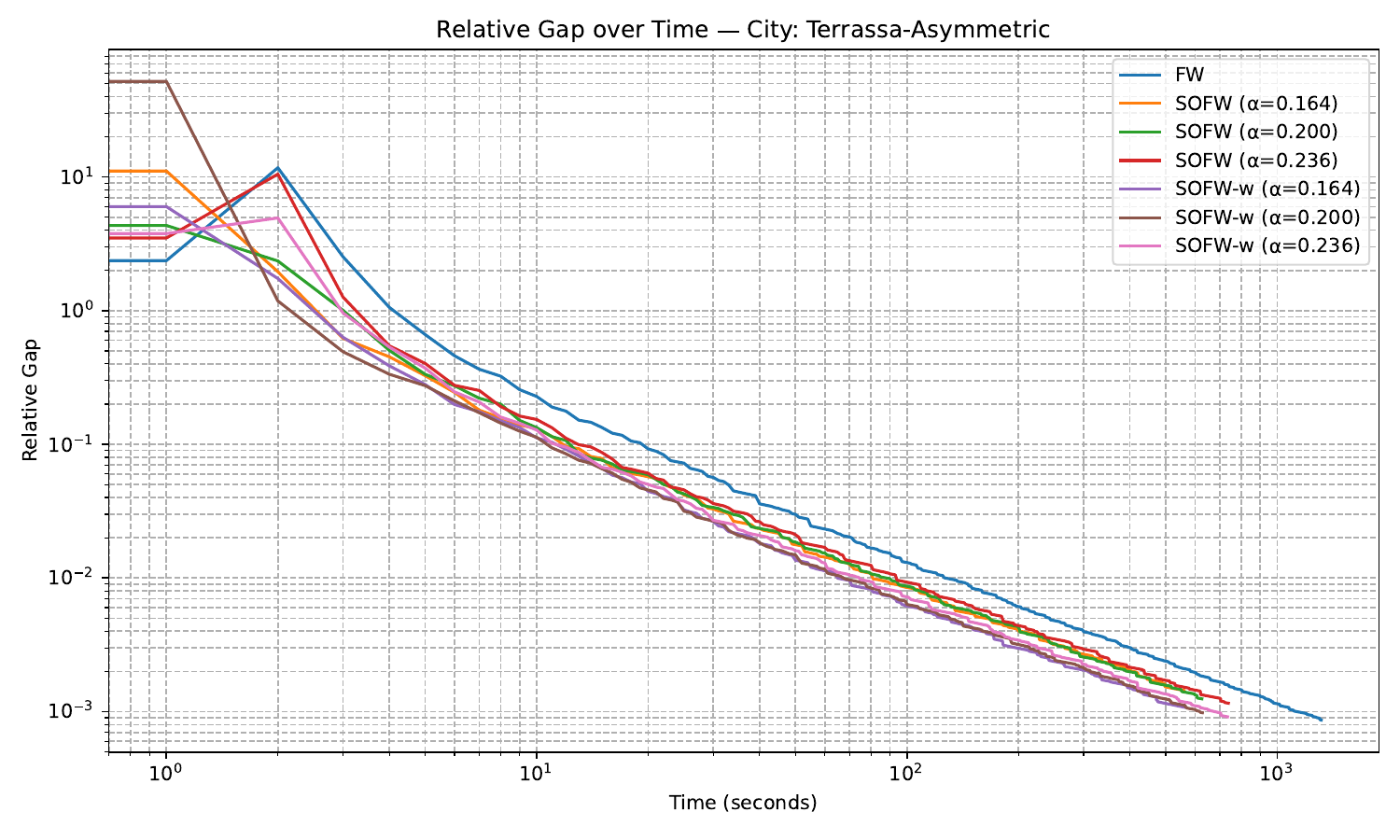}
        \caption{Dataset: Terrassa-Asymmetric. SOFW vs FW on small network.}
    \end{minipage}%
    \begin{minipage}{0.5\textwidth}
        \centering
        \includegraphics[width=\linewidth]{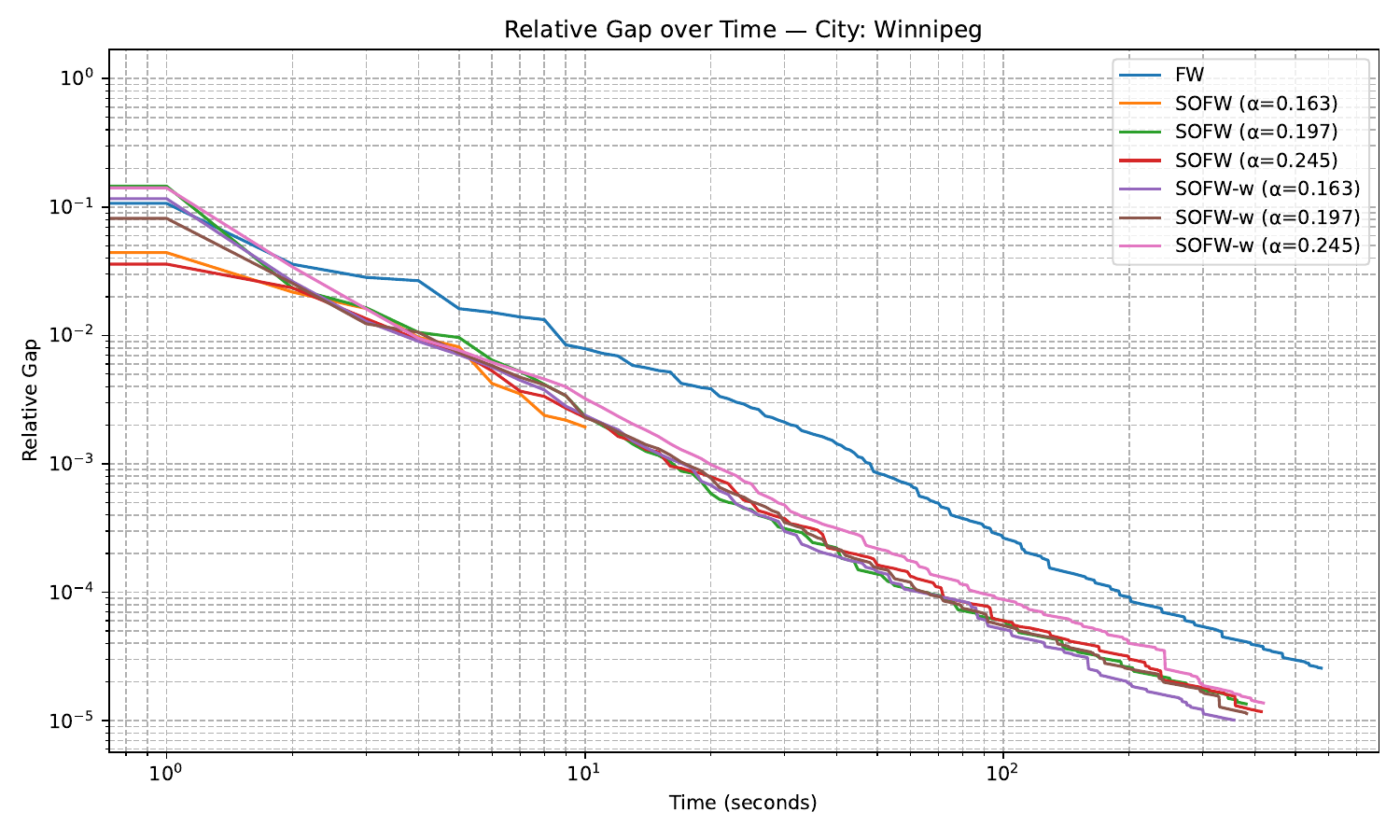}
        \caption{Dataset: Winnipeg. SOFW vs FW on small network.}
    \end{minipage}
    \caption{Comparison of SOFW and FW on small-scale transportation datasets.}
    \label{fig:SOFW-small}
\end{figure}

{\centering
\section{Conclusion}
}

Stochastic approaches demonstrate clear advantages in large-scale transportation modeling, particularly when working with massive datasets typical of large metropolitan areas. In such settings, solving the linear subproblem exactly becomes computationally prohibitive due to the vast number of origin-destination (OD) pairs.

Reducing the problem size by either limiting the number of OD pairs or introducing stochasticity in the selection of origins for shortest path computations provides a practical trade-off between solution quality and computational effort. By sampling a subset of origins at each iteration, it is possible to significantly reduce computational overhead while maintaining satisfactory convergence behavior.

These stochastic schemes enable scalable optimization on mega-scale networks, facilitating efficient and timely traffic assignment and flow estimation in real-world urban systems.

Moreover, it is essential to explore ways of combining stochastic methods with other algorithmic techniques. In particular, an interesting direction for future research is to incorporate stochasticity—such as block-coordinate sampling—into benchmark methods like NFW, potentially yielding more efficient and flexible hybrid algorithms.

In future work, we aim to extend the theoretical analysis to the batched version of the algorithm, which samples and updates multiple sources and their corresponding blocks simultaneously. This extension is expected to reduce the constant factor in the convergence bounds proportionally to the batch size, thereby improving iteration complexity and practical performance.

\noindent

\renewcommand{\refname}{\begin{center}{\Large\bf References} \end{center}}
\makeatletter
\renewcommand{\@biblabel}[1]{#1.\hfill}
\makeatother

\bibliographystyle{amsplain}

\begin{thebibliography}{99}
\providecommand{\url}[1]{\texttt{#1}}
\providecommand{\urlprefix}{URL }
\providecommand{\doi}[1]{https://doi.org/#1}



\fontsize{10.6pt}{4mm}{\selectfont


\vskip4pt\bibitem{babazadeh2020reduced}
Babazadeh, A., Javani, B., Gentile, G., Florian, M.: Reduced gradient algorithm
  for user equilibrium traffic assignment problem. Transportmetrica A:
  Transport Science  \textbf{16}(3),  1111--1135 (2020)

\vskip4pt\bibitem{beckmann1956studies}
Beckmann, M., McGuire, C.B., Winsten, C.B.: Studies in the economics of
  transportation. Tech. rep. (1956)

\vskip4pt\bibitem{blubook-vol1-v091}
Boyles, S.D., Lownes, N.E., Unnikrishnan, A.: Transportation Network Analysis,
  vol.~1 (2023), edition 0.91

\vskip4pt\bibitem{frank1956algorithm}
Frank, M., Wolfe, P., et~al.: An algorithm for quadratic programming. Naval
  research logistics quarterly  \textbf{3}(1-2),  95--110 (1956)

\vskip4pt\bibitem{gasnikov2016numericalbeckman}
Gasnikov, A.V., Dvurechenskii, P., Dorn, Y.V., Maksimov, Y.V.: Numerical
  methods for the problem of traffic flow equilibrium in the beckmann and the
  stable dynamic models. Matematicheskoe modelirovanie  \textbf{28}(10),
  40--64 (2016)

\vskip4pt\bibitem{ignashin2024modifications}
Ignashin, I., Yarmoshik, D.: Modifications of the frank-wolfe algorithm in the
  problem of finding the equilibrium distribution of traffic flows.
  Mathematical Modeling and Numerical Simulation  \textbf{10}(1),  10--25
  (2024)

\vskip4pt\bibitem{jiang2022graph}
Jiang, W., Luo, J.: Graph neural network for traffic forecasting: A survey.
  Expert systems with applications  \textbf{207},  117921 (2022)

\vskip4pt\bibitem{Kerdreux2020}
Kerdreux, T.: Accelerating Frank-Wolfe methods. Ph.D. thesis,
  Universit{\'e} Paris sciences et lettres (2020)

\vskip4pt\bibitem{Kerdreux2018}
Kerdreux~T., Pedregosa~F., d.A.: Frank-wolfe with subsampling oracle. In:
  International Conference on Machine Learning. pp. PMLR 80:2591--2600, 2018.
  (2018)

\vskip4pt\bibitem{kubentayeva2021finding}
Kubentayeva, M., Gasnikov, A.: Finding equilibria in the traffic assignment
  problem with primal-dual gradient methods for stable dynamics model and
  beckmann model. Mathematics  \textbf{9}(11), ~1217 (2021)

\vskip4pt\bibitem{kubentayeva2024primal}
Kubentayeva, M., Yarmoshik, D., Persiianov, M., Kroshnin, A., Kotliarova, E.,
  Tupitsa, N., Pasechnyuk, D., Gasnikov, A., Shvetsov, V., Baryshev, L.,
  et~al.: Primal-dual gradient methods for searching network equilibria in
  combined models with nested choice structure and capacity constraints.
  Computational Management Science  \textbf{21}(1), ~15 (2024)

\vskip4pt\bibitem{jaggi2012blockcoordfw}
Lacoste-Julien, S., Jaggi, M., Schmidt, M., Pletscher, P.: Block-coordinate
  frank-wolfe optimization for structural svms. International Conference on
  Machine Learning  (07 2012)

\vskip4pt\bibitem{levitin1966constrained}
Levitin, E.S., Polyak, B.T.: Constrained minimization methods. USSR
  Computational mathematics and mathematical physics  \textbf{6}(5),  1--50
  (1966)

\vskip4pt\bibitem{mitradjieva2013stiff}
Mitradjieva, M., Lindberg, P.O.: The stiff is moving—conjugate direction
  frank-wolfe methods with applications to traffic assignment. Transportation
  Science  \textbf{47}(2),  280--293 (2013)

\vskip4pt\bibitem{nesterov2003stationary}
Nesterov, Y., De~Palma, A.: Stationary dynamic solutions in congested
  transportation networks: summary and perspectives. Networks and spatial
  economics  \textbf{3},  371--395 (2003)

\vskip4pt\bibitem{perederieieva2015framework}
Perederieieva, O., Ehrgott, M., Raith, A., Wang, J.Y.: A framework for and
  empirical study of algorithms for traffic assignment. Computers \& Operations
  Research  \textbf{54},  90--107 (2015)

\vskip4pt\bibitem{bstabler}
{Transportation Networks for Research Core Team}: Transportation networks for
  research. \url{https://github.com/bstabler/TransportationNetworks} (2024),
  accessed: 2024-02-29

\vskip4pt\bibitem{yarmoshik2024application}
Yarmoshik, D., Persiianov, M.: On the application of saddle-point methods for
  combined equilibrium transportation models. In: International Conference on
  Mathematical Optimization Theory and Operations Research. pp. 432--448.
  Springer (2024)

\vskip4pt\bibitem{zhang2025solving}
Zhang, F., Boyd, S.: Solving large multicommodity network flow problems on
  gpus. arXiv preprint arXiv:2501.17996  (2025)
  
\vskip4pt\bibitem{gasnikov2018universal}
Gasnikov A.V., Nesterov Y.E.:
UNIVERSAL METHOD FOR STOCHASTIC COMPOSITE OPTIMIZATION PROBLEMS
Computational Mathematics and Mathematical Physics. 2018. \textbf{58}(1) pp. 48--64.




}
\end{thebibliography}

\end{document}